\documentclass[a4paper,16pt]{article}
\usepackage[utf8]{inputenc}
\usepackage{amssymb,amsmath,mathrsfs}
\usepackage{amsthm}
\usepackage[colorlinks,
            linkcolor=blue,
            anchorcolor=blue,
            citecolor=green
            ]{hyperref}
\numberwithin{equation}{section}
\usepackage{graphicx}
\usepackage{amsmath}
\newtheorem{theorem}{Theorem}[section]

\theoremstyle{plain}
\newtheorem{thm}[theorem]{Theorem }
\newtheorem{defi}[theorem]{Definition }
\newtheorem{prop}[theorem]{Proposition }
\newtheorem{cor}[theorem]{Corollary }

\newtheorem{lem}[theorem]{Lemma }

\newtheorem{rem}[theorem] {Remark}
\usepackage{dsfont}

\begin{document}

\title{Traveling waves of the quintic focusing NLS--Szeg\H{o} equation}
\author{Ruoci Sun\footnote{Laboratoire de Math\'ematiques d’Orsay, Univ. Paris-Sud \uppercase\expandafter{\romannumeral11}, CNRS, Universit\'e Paris-Saclay, F-91405 Orsay, France (ruoci.sun@math.u-psud.fr).}}

\maketitle

\noindent $\mathbf{Abstract}$ \quad We study the influence of Szeg\H{o} projector $\Pi$ on the $L^2-$critical one-dimensional non linear focusing Schr\"odinger equation, leading to the quintic focusing NLS--Szeg\H{o} equation
\begin{equation*}
i\partial_t u + \partial_x^2 u + \Pi(|u|^4 u)=0, \quad (t,x)\in \mathbb{R}\times \mathbb{R}, \qquad u(0, \cdot)=u_0.
\end{equation*}This equation is globally well-posed in $H^1_+=\Pi(H^1(\mathbb{R}))$, for every initial datum $u_0$. The solution $L^2$-scatters both forward and backward in time if $u_0$ has sufficiently  small mass. We prove the orbital stability with scaling of the traveling wave : $u_{\omega,c}(t,x)=e^{i\omega t}Q(x+ct)$, for some $\omega, c>0$, where $Q$ is a ground state associated to Gagliardo--Nirenberg type functional 
\begin{equation*}
I^{(\gamma)}(f) = \frac{\|\partial_x f\|_{L^2}^2\|f\|_{L^2}^{4}+\gamma \langle-i\partial_x f ,f\rangle_{L^2}^2 \|f\|_{L^2}^2}{\|f\|_{L^{6}}^{6}}, \qquad \forall  f\in H^1_+ \backslash \{0\},
\end{equation*}for some $\gamma\geq 0$. The ground states are completely classified in the case $\gamma=2$, leading to the actual orbital stability without scaling for appropriate traveling waves. As a consequence, the scattering mass threshold of the focusing quintic NLS--Szeg\H{o} equation is strictly below the mass of ground state associated to the functional $I^{(0)}$, unlike the recent result by Dodson $[\ref{Dodson, Global well-posedness and scattering }]$ on the usual quintic focusing non linear Schr\"odinger equation.

\bigskip

\noindent $\mathbf{Keywords}$ \quad $L^2-$critical focusing Schr\"odinger equation, Szeg\H{o} projector, orbital stability, scattering threshold
\tableofcontents

\section{Introduction}

\noindent We consider the following $L^2-$critical 1-dimensional NLS--Szeg\H{o} equation on the line $\mathbb{R}$
\begin{equation}\label{NLS-Szego quintic R}
i\partial_t u + \partial_x^2 u =- \Pi(|u|^4 u) , \quad (t,x)\in \mathbb{R}\times \mathbb{R}, \qquad u(0, \cdot)=u_0,
\end{equation}where $\Pi : L^2(\mathbb{R}) \to L^2(\mathbb{R})$ denotes the Szeg\H{o} projector that cancels all negative Fourier modes
\begin{equation}\label{definition of szego projector on L2 R}
\Pi(f)(x)=  \frac{1}{2\pi} \int_{0}^{+\infty} e^{ix\xi} \hat{f}(\xi)\mathrm{d}\xi, \qquad \forall f \in L^1(\mathbb{R}) \bigcap L^2(\mathbb{R}).
\end{equation}Set $L^2_+=\Pi(L^2(\mathbb{R}))$, then $L^2_+$ can be identified as the Hardy space that consists of all the holomorphic functions on the Poincar\'e half-plane $\mathbb{H}_+:=\{z\in \mathbb{C}: \mathrm{Im}z >0\}$ with $L^2$ boundary
\begin{equation*}
L^2_+=\{f \quad  \mathrm{holomorphic} \quad \mathrm{on} \quad \mathbb{H}_+ : \|f\|_{L^2_+}^2:=\sup_{y >0} \int_{\mathbb{R}} |f(x+iy)|^2 \mathrm{d}x <+ \infty\}.
\end{equation*}Since $\Pi=  \frac{\mathrm{id}+i \mathcal{H}}{2}$, where $\mathcal{H}=-i\mathrm{sign}(-i\partial_x)$ is the Hilbert transform, $\Pi: L^p(\mathbb{R})\to L^p (\mathbb{R})$ is a bounded operator, for every $1<p<+\infty$ (Stein $[\ref{Stein Singular Integrals and Differentiability Properties of Functions}]$). We define the filtered Sobolev spaces $H^s_+=H^s(\mathbb{R}) \bigcap L^2_+$, for every $s\geq 0$. \\

\noindent The motivation to study this equation is based on the following two results. On the one hand, the $L^2-$critical focusing non linear Schr\"odinger equation
\begin{equation}\label{L2 critical focusing NLS equation}
i\partial_t U + \partial_x^2 U = - |U|^4 U, \qquad (t,x) \in \mathbb{R}\times \mathbb{R},\qquad  U(0, \cdot)=U_0.
\end{equation}marks the transition between the global existence (see Cazenave--Weissler $[\ref{Cazenave-weissler1}, \ref{Cazenave-weissler2}]$ for small data case) and the blow-up phenomenon (see Glassey $[\ref{Glassey, R On the blow up of }]$ for viriel identity method, Perelman $[\ref{Perelman On the blow-up phenomenon for the critical nls}]$ and Merle--Raphael $[\ref{Merle--Raphael On universality of blow-up profile L2 critical schrodinger equation}]$ for blow-up dynamics). The instability of traveling waves $U(t,x)=e^{i\omega t}R(x)$ of equation $(\ref{L2 critical focusing NLS equation})$ and the classification of its ground states $R$ associated to the Gagliardo--Nirenberg inequality 
\begin{equation*}
\|f\|_{L^6}^6 \lesssim\|\partial_x f\|_{L^2}^2 \|f\|_{L^2}^4, \qquad \forall f \in H^1(\mathbb{R})
\end{equation*}are established in Weinstein $[\ref{Weinstein nls sharp interpolation estimates}]$. The ground states are unique up to scaling, phase rotation and spatial translation. It has been proved that the scattering mass threshold of equation $(\ref{L2 critical focusing NLS equation})$ is equal to the mass of ground state $\|R\|_{L^2}$ in Dodson $[\ref{Dodson, Global well-posedness and scattering }]$.\\

\noindent On the other hand, it has been shown in G\'erard--Grellier $[\ref{gerardgrellier1}, \ref{Gerard grellier book cubic szego equation and hankel operators}]$ that filtering the positive Fourier modes could accelerate the transition to high frequencies in a Hamiltonian evolution PDE, leading to the super-polynomial growth of Sobolev norms of solutions of the cubic Szeg\H{o} equation on the torus $\mathbb{S}^1$. So we introduce the cubic defocusing NLS--Szeg\H{o} equation on the torus $\mathbb{S}^1$ in Sun $[\ref{Sun Long time behavior of NLS Szego equation}]$ in order to understand how applying a filter keeping only positive Fourier modes modifies the long time dynamics of the non linear Schr\"odinger equation. \\

\noindent We continue this topic in this paper and we put the NLS--Szeg\H{o} equation on the line $\mathbb{R}$. The traveling waves and the classification of ground states of the cubic Szeg\H{o} equation on the line $\mathbb{R}$
\begin{equation}\label{cubic szego equation on R}
i\partial_t V =\Pi(|V|^2 V), \qquad (t,x) \in \mathbb{R}\times\mathbb{R},\qquad  V(0, \cdot)=V_0 
\end{equation}are studied in Pocovnicu $[\ref{pocovnicu Traveling waves for the cubic Szego eq}]$. Now we consider the quintic focusing NLS--Szeg\H{o} equation on the line $\mathbb{R}$ in order to understand how $\Pi$ modifies the global wellposedness result, the scattering mass threshold and the stability result of traveling waves of the $L^2-$critical non linear Schr\"odinger equation. \\

\begin{defi}
Fix $s\geq 0$, a global solution $u\in C(\mathbb{R}; H^s_+)$ of equation $(\ref{NLS-Szego quintic R})$ is said to $H^s-$scatter forward in time if there exists $u_+ \in H^s_+$ such that 
\begin{equation*}
\lim_{t\to +\infty} \|e^{it\partial_x^2} u_+ -u(t)\|_{H^s} =0.
\end{equation*}A global solution $u\in C(\mathbb{R}; H^s_+)$ of equation $(\ref{NLS-Szego quintic R})$ is said to $H^s-$scatter backward in time if there exists $u_- \in H^s_+$ such that 
\begin{equation*}
\lim_{t\to -\infty} \|e^{it\partial_x^2} u_- -u(t)\|_{H^s} =0.
\end{equation*}
\end{defi}

\noindent In the small mass case, equation $(\ref{NLS-Szego quintic R})$ is globally well-posed in $L^2_+$ and the solution $L^2-$scatters both forward and backward in time. The proof is similar to Cazenave--Weissler $[\ref{Cazenave-weissler1}, \ref{Cazenave-weissler2}]$.
\begin{prop}\label{small mass scatter theorem}
There exists $\epsilon_0 >0$ such that if $\|u_0\|_{L^2}\leq \epsilon_0$, then the global solution $u\in C(\mathbb{R}; L^2_+)$ of equation $(\ref{NLS-Szego quintic R})$ exists uniquely and $L^2-$scatters both forward and backward in time.
\end{prop}

\noindent There are three conservation laws for $(\ref{NLS-Szego quintic R})$ and $(\ref{L2 critical focusing NLS equation})$: the mass, the momentum and the Hamiltonian
\begin{equation*}
M(u)=\|u\|_{L^2}^2, \quad P(u)=\langle D u, u \rangle_{L^2},\qquad E(u)=\frac{\|\partial_x u\|_{L^2}^2}{2}  - \frac{\|u\|_{L^6}^6}{6},
\end{equation*}where $D=-i\partial_x$ and $u \in H^1_+$ for $(\ref{NLS-Szego quintic R})$, $u \in H^1(\mathbb{R})$ for $(\ref{L2 critical focusing NLS equation})$. If $u \in C(\mathbb{R}; H^1_+)$ solves equation $(\ref{NLS-Szego quintic R})$, then the momentum $P(u) =\||D|^{\frac{1}{2}}u\|_{L^2}^2$ and the mass $M(u)$ control $H^{\frac{1}{2}}-$norm of the solution, leading to the global wellposedness of equation $(\ref{NLS-Szego quintic R})$. By using Gagliardo--Nirenberg's interpolation inequality 
\begin{equation}\label{Gagliardo-Nirenberg interpolation theta=m/(2m+2)}
\|u\|_{L^{2m+2}} \lesssim_m \||D|^{\frac{1}{2}} u\|_{L^2}^{\frac{m}{m+1}}\|u\|_{L^2}^{\frac{1}{m+1}},  \qquad \forall m \geq 0, \qquad \forall u \in H^{\frac{1}{2}}(\mathbb{R}),
\end{equation}one can solve the  problem of global wellposedness for all $L^2-$supercritical non linear NLS-Szeg\H{o} equations for all large initial data $u_0\in H^1_+$. On the other hand, when $U_0 \in H^1(\mathbb{R})$ such that $E(U_0)<0$ and $ U_0 \in L^2(\mathbb{R}, x^2\mathrm{d}x)$, then the solution $U$ of equation $(\ref{L2 critical focusing NLS equation})$ associated to initial datum $U_0$ blows up in finite time by the viriel identity (Glassey $[\ref{Glassey, R On the blow up of }]$, Cazenave $[\ref{Cazenave introduction to NLS Rio de Janeiro}, \ref{Cazenave book NLS AMS}]$). We refer to Perelman $[\ref{Perelman On the blow-up phenomenon for the critical nls}]$ and Merle--Rapha\"el $[\ref{Merle--Raphael On universality of blow-up profile L2 critical schrodinger equation}]$ to see the asymptotic representation of the blow-up dynamics of equation $(\ref{L2 critical focusing NLS equation})$ in details. Now we state the first result of this paper.\\

\begin{thm}\label{global wellposedness for L2 super critical nls equation}
For all $m \geq 0$, $\lambda=\pm 1$ and $u_0\in H^{1}_+$, there exists a unique function $u\in C(\mathbb{R}; H^1_+)$ solving the following equation
\begin{equation}\label{L2 super critical nls szego}
i\partial_t u + \partial_x^2 u = \lambda \Pi(|u|^{2m} u), \qquad u(0,x)=u_0(x), \qquad (t,x) \in \mathbb{R}\times \mathbb{R}.
\end{equation}
\end{thm}

\noindent The local well-posedness of equation $(\ref{L2 super critical nls szego})$ is established by the fixed-point theorem and Sobolev estimates. In the focusing case $\lambda=-1$, since the mass, the momentum and the Hamiltonian are conserved under the flow of equation $(\ref{L2 super critical nls szego})$, inequality $(\ref{Gagliardo-Nirenberg interpolation theta=m/(2m+2)})$ yields that 
\begin{equation*}
\sup_{t\in \mathbb{R}}\|\partial_x u(t)\|_{L^2}^2 \lesssim_m \|\partial_x u_0\|_{L^2}^2 + \||D|^{\frac{1}{2}}  u_0\|_{L^2}^{2m} \|u_0\|_{L^2}^2.
\end{equation*}Thus every solution of $(\ref{L2 super critical nls szego})$ is global. Besides the global well-posedness problem, there are still other differences between equation $(\ref{NLS-Szego quintic R})$ and equation $(\ref{L2 critical focusing NLS equation})$.\\

\noindent We consider the traveling waves of equation $(\ref{NLS-Szego quintic R})$, $u(t,x)=e^{i\omega t}Q(x+ct)$, for some $\omega, c \in \mathbb{R}$. $u$ solves $(\ref{NLS-Szego quintic R})$ if and only if $Q$ solves the following non local elliptic equation
\begin{equation}\label{equation for bounded state Q}
\partial_x^2 Q + \Pi(|Q|^4 Q) = \omega Q + c D Q. 
\end{equation}It suffices to identity equation $(\ref{equation for bounded state Q})$ to the Euler--Lagrange equation of some functional associated to Gagliardo--Nirenberg inequality $(\ref{Gagliardo-Nirenberg interpolation theta=m/(2m+2)})$ in order to obtain the traveling waves. For all $m\geq 2$ and $\gamma\geq 0$, we define 

\begin{equation}\label{general functional gamma m}
I_m^{(\gamma)}(f) := \frac{\|\partial_x f\|_{L^2}^m\|f\|_{L^2}^{m+2}+\gamma \||D|^{\frac{1}{2}}f\|_{L^2}^{2m}\|f\|_{L^2}^2}{\|f\|_{L^{2m+2}}^{2m+2}}, \qquad \forall  f\in H^1(\mathbb{R}) \backslash \{0\}.
\end{equation}This functional is invariant by space-translation, phase-translation, interior and exterior scaling. 
\begin{equation*}
I_m^{(\gamma)}(f) = I_m^{(\gamma)}(f_{\lambda, \mu, y, \theta}), \qquad \mathrm{where}\qquad f_{\lambda, \mu, y, \theta}(x)=\lambda e^{i\theta} f (\mu x+y), \qquad \forall x, y, \theta\in \mathbb{R}, \qquad \forall \lambda, \mu>0.
\end{equation*}We denote its greatest lower bound by $J_m^{(\gamma)}=\inf_{f\in H^1_+ \backslash \{0\}}I_m^{(\gamma)}(f)$ and all its minimizers by 
\begin{equation}\label{definition of set of ground states}
G_m^{(\gamma)}= \{f \in H^1_+ \backslash \{0\} \quad:\quad I_m^{(\gamma)}(f) =J_m^{(\gamma)}\} = \bigcup_{a,b>0}G_m^{(\gamma)}(a,b),
\end{equation}where $G_m^{(\gamma)}(a,b) = \{f \in G_m^{(\gamma)} : \|f\|_{L^2}=a, \|f\|_{L^6}=b\}$. Then we have
\begin{equation*}
G_m^{(\gamma)}(a,b) =\{\lambda f(\mu\cdot) \in H^1_+ \backslash \{0\}  \quad:\quad f \in G_m^{(\gamma)}(1,1), \quad \lambda=a^{-\frac{1}{m}}b^{\frac{m+1}{m}} \quad \mathrm{and} \quad \mu=b^{\frac{2m+2}{m}} a^{-\frac{2m+2}{m}}\}.
\end{equation*}A concentration-compactness argument shows that the functional $I_m^{(\gamma)}$ attains its minimum in $H^1_+ \backslash \{0\}$. We shall follow the idea of profile decomposition of minimizing sequence introduced in G\'erard $[\ref{gerard profile decomposition}]$, which is a refinement of the concentration-compactness principle (see Lions $[\ref{Lions concentration compactness 1}, \ref{Lions concentration compactness 2}]$ and Cazenave--Lions $[\ref{Cazenave-lions Orbital stability of standing waves }]$ for orbital stability of traveling waves of $L^2-$subcritical NLS equation), in order to establish the existence of minimizers. 

\begin{thm}\label{existence of minimizer of I m gamma general}
For all $m\geq 2$ and $\gamma\geq 0$, if $(f_n)_{n\in \mathbb{N}} \in H^1_+$ is a minimizing sequence for $I_m^{(\gamma)}$ such that $\|f_n\|_{L^2}=\|f_n\|_{L^{2m+2}}=1$ and $\lim_{n\to+\infty}I_m^{(\gamma)}(f_n)= J_m^{(\gamma)}$, then there exists a profile $U  \in G_m^{(\gamma)}(1,1)$, a strictly increasing function $\psi : \mathbb{N}\to \mathbb{N}$ and a real-valued sequence $(x_n)_{n\in \mathbb{N}}$ such that 
\begin{equation}\label{limit of difference between profile and original functions}
\lim_{n\to +\infty}\|f_{\psi(n)} - U(\cdot -x_n)\|_{H^1} = 0.
\end{equation}
\end{thm}

\begin{rem}
If $H^1_+$ is replaced by $H^1(\mathbb{R})$, then we have $U \in H^1(\mathbb{R})\backslash \{0\} $ such that $\|U\|_{L^2}=\|U\|_{L^{2m+2}}=1$, $I_m^{(\gamma)}(U)=\min_{f \in H^1(\mathbb{R}) \backslash \{0\}}I_m^{(\gamma)}(f)$ and the limit $(\ref{limit of difference between profile and original functions})$ also holds.
\end{rem}

\noindent Thus $G_m^{(\gamma)}(a,b)$ is not empty, for all $a, b >0$. We refer to G\'erard--Lenzmann--Pocovnicu--Rapha\"el $[\ref{Gerard Lenzmann Pocovnicu Raphael A two-soliton with transient turbulent regime}]$ to see the asymptotic dynamics and long time behavior in two different regimes of the two-soliton solutions of the cubic focusing half-wave equation on $\mathbb{R}$. Similarly, one obtains the existence of ground states of traveling waves for $L^2-$critical Schr\"odinger equation on $\mathbb{R}^d$ (see Hmidi--Keraani $[\ref{Hmidi--Keraani application}]$), for cubic Szeg\H{o} equation (see Pocovnicu $[\ref{pocovnicu Traveling waves for the cubic Szego eq}]$), for non linear Schr\"odinger equation on the Heisenberg group (see Gassot $[\ref{Gassot radially symmetric traveling waves}]$), etc.\\

\noindent Given $m\geq 2$ and $\gamma\geq 0$, let $f \in H^1_+\backslash \{0\} $ be a minimizer of $I_m^{(\gamma)}$, then $\frac{\mathrm{d}}{\mathrm{d}\epsilon}\Big|_{\epsilon=0} \log I_m^{(\gamma)}(f+\epsilon h)=0$, for all $h\in H^1_+$. $f$ solves the following Euler--Lagrange equation
\begin{equation}\label{Euler Lagrange equation of I m gamma}
\begin{split}
&m\|f\|_{L^2}^{m+2}\|\partial_x f\|_{L^2}^{m-2}\partial_x^2 f +2(m+1)J_m^{(\gamma)}\Pi(|f|^{2m} f)\\
=& ((m+2)\|f\|_{L^2}^{m}\|\partial_x f\|_{L^2}^{m}+2\gamma \||D|^{\frac{1}{2}}f\|_{L^2}^{2m})f + 2\gamma m \|f\|_{L^2}^{2}\||D|^{\frac{1}{2}}f\|_{L^2}^{2m-2} Df.
\end{split}
\end{equation}

\noindent From now on, we restrict ourselves to the case $m=2$. We want to identify equation $(\ref{equation for bounded state Q})$ to equation 
\begin{equation*}
\frac{\|Q\|_{L^2}^{4}}{3J_2^{(\gamma)}}\partial_x^2 Q +\Pi(|Q|^{4} Q)
= \frac{2\|Q\|_{L^2}^2\|\partial_x Q\|_{L^2}^2+\gamma \||D|^{\frac{1}{2}}Q\|_{L^2}^{4}}{3J_2^{(\gamma)}} Q + \frac{2\gamma \|Q\|_{L^2}^{2}\||D|^{\frac{1}{2}}Q\|_{L^2}^{2} }{3J_2^{(\gamma)}} DQ.
\end{equation*}A minimizer $Q^{(\gamma)}\in G_2^{(\gamma)}$ is called the ground state of functional $I_2^{(\gamma)}$, if $\|Q^{(\gamma)}\|_{L^2}^4=  3 J_2^{(\gamma)}$. If $u(t,x)=e^{i\omega t} Q^{(\gamma)}(x+ct)$ solves equation $(\ref{NLS-Szego quintic R})$ and $Q^{(\gamma)} \in G_2^{(\gamma)}(\sqrt[4]{3 J_2^{(\gamma)}}, b)$ for some $b,\omega>0$ and $c, \gamma \geq 0$, then $c=0$ if and only if $\gamma=0$.\\

\noindent If $\gamma=0$, then we have $\|\partial_x Q^{(\gamma)}\|_{L^2}^2 = \frac{\omega}{2} \sqrt{3 J_2^{(0)}}$ and $\| Q^{(\gamma)}\|_{L^6}^6 =  \frac{3\omega}{2}\sqrt{3 J_2^{(0)}} $.

\noindent If $\gamma>0$, then we have
\begin{equation*}
\|\partial_x Q^{(\gamma)}\|_{L^2}^2=\frac{\sqrt{3J_2^{(\gamma)}}}{8\gamma}(4\gamma \omega -c^2),\qquad \||D|^{\frac{1}{2}} Q^{(\gamma)}\|_{L^2}^2 =\frac{\sqrt{3J_2^{(\gamma)}}c}{2\gamma},\qquad \| Q^{(\gamma)}\|_{L^6}^6 = 3 \sqrt{3J_2^{(\gamma)}} (\frac{\omega}{2}+ \frac{c^2}{8\gamma}).
\end{equation*}Furthermore, the interpolation inequality $\||D|^{\frac{1}{2}} Q^{(\gamma)}\|_{L^2}^2 \leq \| Q^{(\gamma)}\|_{L^2}\|\partial_x Q^{(\gamma)}\|_{L^2}$ yields that $c^2 \leq \frac{4 \gamma^2 \omega}{\gamma+2}$.\\

\noindent Even though we do not know how to classify all the ground states of $I_2^{(\gamma)}$, for general $\gamma\geq 0$, the $H^1-$orbital stability with scaling can be established by using theorem $\ref{existence of minimizer of I m gamma general}$ and the conservation law $P(u)=\langle -i \partial_x u, u\rangle_{L^2} = \||D|^{\frac{1}{2}} u\|_{L^2}^2$. 

\begin{thm}\label{Orbital stability of ground state of I m=2 gamma positif}
For every $\epsilon, b>0$ and $\gamma\geq 0$, there exists $\delta=\delta(b, \epsilon, \gamma) >0$ such that if
\begin{equation*}
\inf_{f \in G_2^{(\gamma)}(\sqrt[4]{3J_2^{(\gamma)}}, b)} \|u_0 - f\|_{H^1}<\delta,
\end{equation*}then we have $\sup_{t \in \mathbb{R}}\inf_{\Psi \in \bigcup_{C(\gamma)^{-1}\leq \theta \leq C(\gamma)} G_2^{(\gamma)}(\sqrt[4]{3J_2^{(\gamma)}}, \theta b)}\|u(t) - \Psi\|_{H^1}<\epsilon$, where $u$ is the solution of equation $(\ref{NLS-Szego quintic R})$ with initial datum $u(0)=u_0$ and $C(\gamma):=\left(\inf_{f \in H^1_+\backslash \{0\}} \frac{\||D|^{\frac{1}{3}} f\|_{L^2}}{\|f\|_{L^6}}\right)^{-1}\sqrt[6]{\frac{J_2^{(\gamma)}}{1+\gamma}}$.
\end{thm}

\noindent The problem of uniqueness of ground states of a non-local elliptic equation is difficult (See Frank--Lenzmann $[\ref{Frank Lenzmann Uniqueness of non linear ground states}]$ for the fractional Laplacians in $\mathbb{R}$ and also Lenzmann--Sok $[\ref{Lenzmann sok rearrangement of fourier modes}]$ for a strict rearrangement principle in Fourier space). We refer to subsection $\ref{subsection of Problem of uniqueness of ground states for general gamma}$ to discuss the classification of ground states of $I^{(\gamma)}_2$. For general $\gamma\geq 0$, we only have the $H^1-$orbital stability with scaling. Since we do not know the uniqueness of ground states of $I_2^{(\gamma)}$, the $L^6-$norm of the ground state $\Psi$ that approaches $u(t)$ is unknown. We can only give a range $\frac{\|f\|_{L^6}}{C(\gamma)}  \leq \|\Psi\|_{L^6} \leq C(\gamma)\|f\|_{L^6}$, where $f$ denotes the ground state that approaches the initial datum $u_0$.\\

\noindent On the other hand, all the ground states can be completely classified in the case $\gamma=2$, by using Cauchy--Schwarz inequality. The ground state of $I_2^{(2)}$ is unique up to scaling, phase rotation and spatial translation. 

\begin{prop}\label{Cauchy Schwarz estimate to find minimizer}
In the case $\gamma=m=2$, we have
\begin{equation}\label{classification of the case m=2}
J_2^{(2)}=\min_{f\in H^1_+ \backslash \{0\}}I_2^{(2)}(f)=\frac{8\pi^2}{3}, \qquad G_2^{(2)} = \{x\mapsto\frac{\lambda e^{i\theta}}{\mu x+ y +i} \in H^1_+: \forall \lambda, \mu >0 \quad \forall \theta, y \in \mathbb{R}\}.
\end{equation} 
\end{prop}

\noindent If $u(t,x)=e^{i\omega t}Q(x+ct)$ is a traveling wave of $(\ref{NLS-Szego quintic R})$ and $Q \in G_2^{(2)}$, then we have 
\begin{equation*}
3c^2=8 \omega, \quad\|Q\|_{L^2}^4=8\pi^2, \quad \|Q\|_{L^6}^6 = \frac{3\pi c^2}{\sqrt{2}}, \quad \||D|^{\frac{1}{2}}Q\|_{L^2}^2=\frac{\pi c}{\sqrt{2}}, \quad  \|\partial_x Q\|_{L^2}^2=\frac{\pi c^2}{2\sqrt{2}},
\end{equation*}thanks to the classification of $G^{(2)}_2$. In this case, the traveling wave $u_c(t,x)=e^{\frac{3c^2 i t}{8}}Q_c(x+ct)$ is $H^1-$orbitally stable, for every $c>0$, where $Q_c \in  G^{(2)}_2(\sqrt[4]{8\pi^2}, \sqrt[6]{\frac{3\pi c^2}{\sqrt{2}}})$.

\begin{thm}\label{Orbital stability of ground state of I m=2 gamma=2}
For every $\epsilon, c>0$, there exists $\delta_{\epsilon,c}>0$ such that if $\inf_{f\in \tilde{G}(\sqrt[4]{8\pi^2}, \sqrt[6]{\frac{3\pi c^2}{\sqrt{2}}})}\|u_0-f\|_{H^1} < \delta$, then we have $\sup_{t \in \mathbb{R}}\inf_{f\in \tilde{G}(\sqrt[4]{8\pi^2}, \sqrt[6]{\frac{3\pi c^2}{\sqrt{2}}})}\|u(t)-f\|_{H^1} < \epsilon$, where $u$ solves equation $(\ref{NLS-Szego quintic R})$ with initial datum $u(0)=u_0$.
\end{thm}

\begin{rem}
In the case $\gamma=2$, since all the ground states are completely classified and unique up to scaling, phase rotation and spatial translation by proposition $\ref{Cauchy Schwarz estimate to find minimizer}$, we obtain the $\dot{H}^{\frac{1}{2}}-$norm of the ground state that approaches $u(t)$ by the conservation law $P(u)=\langle Du, u\rangle_{L^2}$ and formula $(\ref{limit of difference between profile and original functions})$. Then we know also the $L^6-$norm of the very ground state and we have the actual orbital stability without scaling, which is a refinement of theorem $\ref{Orbital stability of ground state of I m=2 gamma positif}$.
\end{rem}

\noindent Similar results on the classification of ground states of traveling waves by using Cauchy--Schwarz inequality can be found in Foschi $[\ref{Foschi, Maximizers for the Strichartz Inequality}]$ for linear Schr\"odinger equation and linear wave equation on $\mathbb{R}^d$, for $d \geq 1$, G\'erard--Grellier $[\ref{gerardgrellier1}, \ref{Gerard grellier book cubic szego equation and hankel operators}]$ for the cubic Szeg\H{o} equation on the torus $\mathbb{S}^1$, Pocovnicu $[\ref{pocovnicu Traveling waves for the cubic Szego eq}]$ for the cubic Szeg\H{o} equation on the line $\mathbb{R}$ and Gassot $[\ref{Gassot radially symmetric traveling waves}]$ for the non linear Schr\"odinger equation on the Heisenberg group.\\

\noindent In the case $\gamma =0$, we have 
\begin{equation*}
I_2^{(0)}(f):=\frac{\|\partial_x f\|_{L^2}^2 \| f\|_{L^2}^4}{\| f\|_{L^6}^6}, \qquad \forall f \in H^1(\mathbb{R})\backslash \{0\}.
\end{equation*}All of the ground states in $H^1(\mathbb{R})$ of $I_2^{(0)}$ have been completely classified in Weinstein $[\ref{Weinstein nls sharp interpolation estimates}]$. We know $\min_{f \in H^1(\mathbb{R})\backslash \{0\}} I_2^{(0)}(f)=\frac{\pi^2}{4}$ and there exists a unique real-valued, positive, spherically symmetric and decreasing function $R(x)=\frac{\sqrt[4]{3}}{\sqrt{\cosh(2x)}}$ such that
\begin{equation*}
(I_2^{(0)})^{-1}(\frac{\pi^2}{4})=\{\lambda  e^{i\theta} R(\mu \cdot -y) : \lambda, \mu>0 \quad \theta, y \in \mathbb{R}\}.
\end{equation*}The traveling wave $U(t,x)=e^{i\omega t} R(x)$ is an unstable solution of the $L^2-$critical focusing Sch\"odinger equation $(\ref{L2 critical focusing NLS equation})$ in the following sense: there exists a sequence $u_0^{(n)}=(1+\frac{1}{n})R \subset H^1(\mathbb{R})$ such that $u_0^{(n)} \to R$, as $n \to +\infty$, but the corresponding maximal solution $u^{(n)}$ blows up in finite time. We denote by $Q^{(0)} \in G_2^{(0)}(\sqrt[4]{3J_2^{(0)}},\|Q^{(0)}\|_{L^6})$ one of the ground states of $I_2^{(0)}$ in Hardy space $H^1_+$. Since $R\notin H^1_+$ and $Q_+: x\mapsto \frac{1}{x+i} \in H^1_+$, we have $\frac{\pi^2}{4}=I_2^{(0)}(R) < J_2^{(0)}=I_2^{(0)}(Q^{(0)})  \leq I_2^{(0)}(Q_+)=\frac{4\pi^2}{3}$.\\

\noindent In proposition $\ref{small mass scatter theorem}$, we know that the solution of equation $(\ref{NLS-Szego quintic R})$ and equation $(\ref{L2 critical focusing NLS equation})$ scatters if the initial datum has sufficiently small mass. Furthermore, Dodson $[\ref{Dodson, Global well-posedness and scattering }]$ has proved that if $\|U_0\| < \|R\|_{L^2}$, then equation $(\ref{L2 critical focusing NLS equation})$ is globally well-posed and the solution $L^2-$scatters both forward and backward in time. Together with the instability result of traveling waves by Weinstein $[\ref{Weinstein nls sharp interpolation estimates}]$, the scattering mass threshold of equation $(\ref{L2 critical focusing NLS equation})$ is equal to the mass of ground state $R\in H^1(\mathbb{R})$ of $I_2^{(0)}$.\\

\noindent On the other hand, adding the Szeg\H{o} projector in front of the non linear term of the $L^2-$critical focusing Schr\"odinger equation makes the scattering mass threshold strictly less than the mass of ground state $Q^{(0)} \in H^1_+$ of $I_2^{(0)}$, thanks to the orbital stability theorem $\ref{Orbital stability of ground state of I m=2 gamma positif}$. We define $\mathcal{E} \subset \mathbb{R}_+^*$ to be all $\epsilon>0$ such that if $\|u_0\|_{L^2}<\epsilon$, the corresponding solution of $(\ref{NLS-Szego quintic R})$ $L^2-$scatters both forward and backward in time. If an $H^1-$solution of equation $(\ref{NLS-Szego quintic R})$ $L^2-$scatters, then it also $H^1-$scatters and its $L^r-$norm decays, with $2<r\leq +\infty$. Thus traveling waves do not $L^2-$scatter and neither does the solution that approaches the traveling wave.

\begin{cor}\label{scattering threshold reduced}
$\sup \mathcal{E} < \|Q^{(0)}\|_{L^2}= \sqrt[4]{3J_2^{(0)}}$. Precisely, there exists $u\in C(\mathbb{R}; H^1_+)$ solving equation $(\ref{NLS-Szego quintic R})$ such that $\|u(0)\|_{L^2}< \|Q^{(0)}\|_{L^2}$ and $u$ does not $L^2-$scatter neither forward nor backward in time.
\end{cor}

\begin{rem}
The mass of ground state of $I_2^{(2)}$ is strictly larger than the mass of ground state of $I_2^{(0)}$. $\|Q^{(2)}\|_{L^2}^4=8\pi^2 > 4\pi^2\geq \|Q^{(0)}\|_{L^2}^4$. 
\begin{equation*}
E(Q^{(2)}) = -\frac{\pi c^2}{4\sqrt{2}} <0 = E(Q^{(0)}).
\end{equation*}
\end{rem}

\noindent The value of scattering mass threshold of equation $(\ref{NLS-Szego quintic R})$ remains as an open problem. Due to lack of Galilean invariance, we cannot use the Morawetz estimate here. We do not know whether the scattering mass threshold of equation $(\ref{NLS-Szego quintic R})$ is equal to the scattering mass threshold of equation $(\ref{L2 critical focusing NLS equation})$. \\

\noindent This paper is organized as follows. In section $\ref{Profile decomposition section}$, we recall profile decomposition theorem prove theorem $\ref{existence of minimizer of I m gamma general}$. In section $\ref{Traveling wave of ground state section}$, the orbital stability of traveling wave $u(t,x)=e^{i \omega t}Q(x+ct)$ is proved at first. Then we give the details of the special case $\gamma=2$. In the first appendix, we prove the persistence of regularity of scattering. In the second appendix, we discuss the open problem of uniqueness of ground states of the functional $I_2^{(\gamma)}$, for general $\gamma \geq 0$. It suffices to study the uniqueness of ground states modulo positive Fourier transform.

\bigskip
\bigskip

\begin{center}
$Acknowledgments$
\end{center}
The author would like to express his gratitude towards Patrick G\'erard for his deep insight, generous advice and continuous encouragement.

\bigskip
\bigskip

\section{Profile decomposition}\label{Profile decomposition section}
At first, we recall the result of profile decomposition theorem in G\'erard $[\ref{gerard profile decomposition}]$, which is a refinement of concentration-compactness argument of Sobolev embedding introduced in Lions $[\ref{Lions concentration compactness 1}, \ref{Lions concentration compactness 2}]$. Every bounded sequence in $H^1_+$ has a subsequence which can be written as a nearly orthogonal sum of a superposition of sequence of shifted profiles and a sequence tending to zero in $L^p (\mathbb{R})$, for every $2<p \leq +\infty$. It will be used to find the minimizers of some functionals in calculus of variation and establish the orbital stability of some traveling waves. We shall use the version of Hmidi--Keraani $[\ref{Hmidi--Keraani Blowup theory profile decomposition}]$ and  construct the profiles without scaling.
\begin{thm}[G\'erard \ref{gerard profile decomposition}, Hmidi--Keraani \ref{Hmidi--Keraani Blowup theory profile decomposition}]\label{Profile decomposition H1+}
If $(f_n)_{n\in \mathbb{N}_+}$ is a bounded sequence in $H^1_+$, then there exists a subsequence of $(f_n)_{n\in \mathbb{N}_+}$, denoted by $(f_{\phi(n)})_{n\in \mathbb{N}_+}$, a sequence of profiles $(U^{(j)})_{j\in \mathbb{N}_+}\subset H^1_+$ and a double-indexed sequence $(x_n^{(j)})_{n,j\in \mathbb{N}_+}\subset \mathbb{R}$ such that if $j\ne k$, then $|x_n^{(j)}-x_n^{(k)}| \to +\infty$ and for every $l \in \mathbb{N}_+$, we have
\begin{equation}\label{profile decomposition formula}
f_{\phi(n)}(x) = \sum_{j=1}^l U^{(j)}(x -x_n^{(j)}) + r_n^{(l)}(x), \qquad \forall x \in \mathbb{R},
\end{equation}where $\limsup_{n\to+\infty}\|r_n^{(l)}\|_{L^p} \to 0$, as $l \to +\infty$, $\forall 2<p \leq +\infty$. For every $l\in \mathbb{N}_+$ and $s \in [0,1]$,  
\begin{equation}\label{near orthogonality dot(H) s}
\Big| \||D|^{s} f_{\phi(n)}\|_{L^2}^2-\sum_{j=1}^l \||D|^{s}  U^{(j)}\|_{L^2}^2 -\||D|^{s}  r_n^{(l)}\|_{L^2}^2 \Big| \to 0, \qquad \mathrm{as} \qquad n \to +\infty.
\end{equation}
\end{thm}

\begin{rem}\label{U 1 =0 then every U is 0}
According to G\'erard $[\ref{gerard profile decomposition}]$ and Hmidi--Keraani $[\ref{Hmidi--Keraani Blowup theory profile decomposition}]$, we may construct the profiles $(U^{(j)})_{j\in \mathbb{N}_+}\subset H^1_+$ such that if $U^{(l)}=0$ for some $l\in \mathbb{N}$, then $U^{(j)}=0$, for every $j\geq l$.
\end{rem}

\noindent Then we shall use this theorem to establish the existence of minimizers in $H^1_+$ of $I_m^{(\gamma)}$, for every $m\geq 2$ and $\gamma\geq 0$. Similar applications may be found in Hmidi--Keraani $[\ref{Hmidi--Keraani application}]$ Pocovnicu $[\ref{pocovnicu Traveling waves for the cubic Szego eq}]$ and Gassot $[\ref{Gassot radially symmetric traveling waves}]$.

\begin{proof}[Proof of theorem $\ref{existence of minimizer of I m gamma general}$]
Since $\sup_{n\in \mathbb{N}}\|f_n\|_{H^1} <+\infty$, theorem $\ref{Profile decomposition H1+}$ gives a subsequence of $(f_n)_{n\in \mathbb{N}_+}$, denoted by $(f_{\phi(n)})_{n\in \mathbb{N}_+}$, a sequence of profiles $(U^{(j)})_{j\in \mathbb{N}_+}\subset H^1_+$ and a double-indexed sequence $(x_n^{(j)})_{n,j\in \mathbb{N}_+}\subset \mathbb{R}$ such that $|x_n^{(j)}-x_n^{(k)}| \to +\infty$, if $j\ne k$, $(\ref{profile decomposition formula})$ and $(\ref{near orthogonality dot(H) s})$ hold and $\lim_{l \to +\infty}\limsup_{n\to+\infty}\|r_n^{(l)}\|_{L^{2m+2}} \to 0$.\\

\noindent For all $0\leq s \leq 1$, $l\in \mathbb{N}_+$ and $\delta >0$, there exists $N=N(s,l,\delta) \in \mathbb{N}_+$ such that
\begin{equation*}
\||D|^{s}f_{\phi(n)}\|_{L^2}^2 \geq \sum_{j=1}^l \||D|^{s} U^{(j)}\|_{L^2}^2 -\delta, \qquad \forall n >N.
\end{equation*}Taking $n\to +\infty$, $\delta\to 0$ and $l\to +\infty$, we have $\sum_{j=1}^l \||D|^{s} U^{(j)}\|_{L^2}^2 \leq \liminf_{n\to+\infty}\||D|^{s} f_{\phi(n)}\|_{L^2}^2$, for every $0\leq s \leq 1$. Then, there exists a subsequence of $(f_{\phi(n)})_{n\in \mathbb{N}}$, denoted by $(f_{\phi \circ \tilde{\phi}(n)})_{n\in \mathbb{N}}$ such that both sequences $(\||D|^{\frac{1}{2}}f_{\phi \circ \tilde{\phi}(n)}\|_{L^2})_{n\in \mathbb{N}}$ and $(\|\partial_x f_{\phi \circ \tilde{\phi}(n)}\|_{L^2})_{n\in \mathbb{N}}$ converge and we have
\begin{equation}\label{estimate of sum of profiles}
\begin{cases}
&\sum_{j=1}^{+\infty}\|\partial_x U^{(j)}\|_{L^2}^2 \leq \lim_{n\to +\infty}\|\partial_x f_{\phi \circ \tilde{\phi}(n)}\|_{L^2}^2,\\ 
&\sum_{j=1}^{+\infty}\||D|^{\frac{1}{2}} U^{(j)}\|_{L^2}^2 \leq \lim_{n\to +\infty}\||D|^{\frac{1}{2}}  f_{\phi \circ \tilde{\phi}(n)}\|_{L^2}^2,\\
&\sum_{j=1}^{+\infty}\|U^{(j)}\|_{L^2}^2 \leq \lim_{n\to +\infty}\| f_{\phi(n)}\|_{L^2}^2 = 1.
\end{cases}
\end{equation}Thus $0\leq \|U^{(j)}\|_{L^2}\leq 1$, for every $j\in \mathbb{N}_+$. We set $\psi=\phi \circ \tilde{\phi} : \mathbb{N \to \mathbb{N}}$. Since $m\geq 2$ and
\begin{equation*}
J_m^{(\gamma)}= \lim_{n\to +\infty}I_m^{(\gamma)}(f_{\psi(n)}) =\lim_{n\to +\infty}\|\partial_x f_{\psi(n)}\|_{L^2}^m + \gamma \lim_{n\to +\infty}\||D|^{\frac{1}{2}}  f_{\psi(n)}\|_{L^2}^{2m},
\end{equation*}estimates $(\ref{estimate of sum of profiles})$ yields that \begin{equation}\label{J_m bigger than Jm times sum}
\begin{split}
J_m^{(\gamma)} \geq &(\sum_{j=1}^{+\infty}\|\partial_x U^{(j)}\|_{L^2}^2)^{\frac{m}{2}}+ \gamma (\sum_{j=1}^{+\infty}\||D|^{\frac{1}{2}} U^{(j)}\|_{L^2}^2)^{m}\\
 \geq & \sum_{j=1}^{+\infty}(\|\partial_x U^{(j)}\|_{L^2}^m \|U^{(j)}\|_{L^2}^{m+2} + \gamma \||D|^{\frac{1}{2}} U^{(j)}\|_{L^2}^{2m} \|U^{(j)}\|_{L^2}^{2})\\
 \geq & J_m^{(\gamma)} \sum_{j=1}^{+\infty}\|U^{(j)}\|_{L^{2m+2}}^{2m+2}.
\end{split}
\end{equation}We claim that 
\begin{equation}\label{sum of L(2m+2) norm =1}
\sum_{j=1}^{+\infty}\| U^{(j)}\|_{L^{2m+2}}^{2m+2}=1.
\end{equation}In fact, each profile $U^{(j)} \in L^{\infty}(\mathbb{R})$. More precisely, we have $\sum_{j=1}^{\infty} \| U^{(j)}\|_{L^{\infty}}^2 \lesssim_m 1$ by Sobolev embedding $H^1 \hookrightarrow L^{\infty}$ and estimate $(\ref{estimate of sum of profiles})$. One can easily check that 
\begin{equation*}
\|\sum_{j=1}^{l}U^{(j)}\|_{L^{2m+2}}^{2m+2} = \sum_{j=1}^{l}\|U^{(j)}\|_{L^{2m+2}}^{2m+2} +R_n^{(l)}, \qquad \forall l \in \mathbb{N}_+,
\end{equation*}where $|R_n^{(l)}|\lesssim_{m,l} \sum_{1\leq j < k\leq l} \int_{\mathbb{R}}|U^{(j)}(x-x_n^{(j)})||U^{(k)}(x-x_n^{(k)})|\mathrm{d}x$. Since $U^{(j)} \in L^2_+$, $\lim_{n\to +\infty}|x_n^{(j)}-x_n^{(k)}| = + \infty$, for all $1 \leq j <k \leq l$, we have $|U^{(j)}(\cdot -x_n^{(j)}+x_n^{(k)})|\rightharpoonup 0$ in $L^2_+$, as $n\to +\infty$. So $\lim_{n\to+\infty}R_n^{(l)} =0$. The profile decomposition theorem $(\ref{profile decomposition formula})$ implies that 
\begin{equation*}
\|r_n^{(l)}\|_{L^{2m+2}} \geq \Big| \|f_{\psi(n)}\|_{L^{2m+2}} - \|\sum_{j=1}^{l}U^{(j)}\|_{L^{2m+2}}  \Big| = \Big| 1 - (\sum_{j=1}^{l}\|U^{(j)}\|_{L^{2m+2}}^{2m+2} +R_n^{(l)})^{\frac{1}{2m+2}} \Big|.
\end{equation*}Taking $n \to +\infty$, we have $\Big| 1 - (\sum_{j=1}^{l}\|U^{(j)}\|_{L^{2m+2}}^{2m+2} )^{\frac{1}{2m+2}} \Big| \leq  \limsup_{n \to +\infty}\|r_n^{(l)}\|_{L^{2m+2}} \to 0$, as $l \to +\infty$ and we obtain $(\ref{sum of L(2m+2) norm =1})$.\\

\noindent Combining $(\ref{J_m bigger than Jm times sum})$, we have $J_m^{(\gamma)} \geq J_m^{(\gamma)} \sum_{j=1}^{+\infty}\| U^{(j)}\|_{L^{2m+2}}^{2m+2} =J_m^{(\gamma)}$. All inequalities in $(\ref{estimate of sum of profiles})$ and $(\ref{J_m bigger than Jm times sum})$ are actually equalities. In particular, we have 
\begin{equation*}
\|\partial_x U^{(1)}\|_{L^2}^m \|U^{(1)}\|_{L^2}^{2m+2}= \|\partial_x U^{(1)}\|_{L^2}^m , \qquad I_m^{(\gamma)}(U^{(1)})\|U^{(1)}\|_{L^{2m+2}}^{2m+2} =J_m^{(\gamma)}\|U^{(1)}\|_{L^{2m+2}}^{2m+2}.
\end{equation*}If $U^{(1)} \equiv 0$ a.e. in $\mathbb{R}$, then so is $U^{(j)}$, for every $j\geq 2$ by the construction of profiles in remark $\ref{U 1 =0 then every U is 0}$. It contradicts formula $(\ref{sum of L(2m+2) norm =1})$. Thus Gagliardo--Nirenberg inequality yields that $\|\partial_x U^{(1)}\|_{L^2} >0$ and $\|U^{(1)}\|_{L^2}=1$. Since $\sum_{j=1}^{+\infty} \|U^{(j)}\|_{L^2}^{2} \leq 1$, we have $U^{(j)}=0$ a.e., for every $j\geq 2$. Formula $(\ref{sum of L(2m+2) norm =1})$ implies that $\|U^{(1)}\|_{L^{2m+2}}=1$ and $I_m^{(\gamma)}(U^{(1)}) =J_m^{(\gamma)}$.\\

\noindent Estimate $(\ref{estimate of sum of profiles})$ is also equality, so $ \|\partial_x U^{(1)}\|_{L^2}^2 = \lim_{n\to+\infty} \|\partial_x f_{\psi(n)}\|_{L^2}^2$.  We set
\begin{equation*}
\epsilon_{n,s}^{(l)}:=\||D|^{s}f_{\psi(n)}\|_{L^2}^2-\sum_{j=1}^l \| |D|^{s} V^{(j)}\|_{L^2}^2 -\||D|^{s} r_{\tilde{\phi}(n)}^{(l)}\|_{L^2}^2, \qquad \forall n \in \mathbb{N}, \quad l\in \mathbb{N}_+, \quad 0\leq s\leq 1.
\end{equation*}Thus $\lim_{n\to+\infty}\epsilon_{n,s}^{(l)} =0$, for all $l\in \mathbb{N}_+$, $0\leq s\leq 1$. Then
\begin{equation*}
\|f_{\psi(n)} - U^{(1)}(\cdot - y_n^{(1)})\|_{H^1}^2 = \|r_{\tilde{\phi}(n)}^{(1)}\|_{H^1}^2 =  \|\partial_x f_{\psi(n)}\|_{L^2}^2 -\|\partial_x U^{(1)}\|_{L^2}^2 -\epsilon_{n,0}^{(1)}-\epsilon_{n,1}^{(1)}  \to 0,
\end{equation*}as $n\to +\infty$.
\end{proof}

\bigskip
\bigskip

\section{Orbital stability of traveling wave $u(t,x)=e^{i\omega t}Q(x+ct)$}\label{Traveling wave of ground state section}
\noindent At first, we prove the $H^1-$orbital stability with scaling for traveling wave $u(t,x)=e^{i\omega t}Q(x+ct)$ of the mass-critical focusing NLS-Szeg\H{o} equation,
\begin{equation*}
i\partial_t u + \partial_x^2 u =- \Pi(|u|^4 u) , \quad (t,x)\in \mathbb{R}\times \mathbb{R}, \qquad u(0, \cdot)=u_0,
\end{equation*}with $\omega,c>0$, $\gamma\geq 0$ and $Q \in G_2^{(\gamma)}(\sqrt[4]{3J_2^{(\gamma)}}, \|Q\|_{L^6})$. Then we focus on the case $\gamma=2$ by classifying all ground states and refining theorem $\ref{Orbital stability of ground state of I m=2 gamma positif}$ as theorem $\ref{Orbital stability of ground state of I m=2 gamma=2}$.

\bigskip

\subsection{Proof of theorem $\ref{Orbital stability of ground state of I m=2 gamma positif}$}
\begin{proof} 
Fix $b>0$ and $\gamma\geq 0$, for every $n \in \mathbb{N}$, we choose $u_0^n\in H^1_+$ and $\varphi^n \in G_2^{(\gamma)}(\sqrt[4]{3J_2^{(\gamma)}},  b)$ such that $\|u_0^n-\varphi^n\|_{H^1} \to 0$, as $n \to +\infty$. Let $u^n$ solve $(\ref{NLS-Szego quintic R})$ with initial datum $u^n(0)=u^n_0$. We shall prove that
\begin{equation}\label{goal of orbital stability alpha =0}
\inf_{\Psi \in \bigcup_{C(\gamma)^{-1}\leq \theta \leq C(\gamma)} G_2^{(\gamma)}(\sqrt[4]{3J_2^{(\gamma)}}, \theta b)}\|u^n(t^n) - \Psi\|_{H^1}\to 0,  \qquad \mathrm{as} \qquad n \to +\infty,
\end{equation}up to a subsequence, for every temporal sequence $(t^n)_{n\in \mathbb{N}} \subset \mathbb{R}$. We use the three conservation laws  
\begin{equation*}
E(u)=\frac{\|\partial_x u\|_{L^2}^2 }{2}-\frac{\|u\|_{L^6}^6}{6}, \qquad P(u)=\langle D u, u \rangle_{L^2} = \||D|^{\frac{1}{2}} u\|_{L^2}^2, \qquad M(u)=\|u\|_{L^2}^2.
\end{equation*}to construct another conservation law
\begin{equation}\label{new conservation law K gamma}
K_{\gamma}(u):=E(u)+ \frac{\gamma P(u)^2}{M(u)}=\frac{\|\partial_x u\|_{L^2}^2 }{2}-\frac{\|u\|_{L^6}^6}{6} + \frac{\gamma \||D|^{\frac{1}{2}} u\|_{L^2}^4}{2 \|u\|_{L^2}^2} = \frac{\|u\|_{L^6}^6}{6\|u\|_{L^2}^4}(3 I_2^{(\gamma)}(u)-\|u\|_{L^2}^4).
\end{equation}Since $I_2^{(\gamma)}(\varphi^n)=J_2^{(\gamma)}$ and $\|u_0^n-\varphi^n\|_{H^1} \to 0$, as $n \to +\infty$, we have $\lim_{n \to +\infty}\|u_0^n\|_{L^2}^4 =3J_2^{(\gamma)}$ and $\lim_{n \to +\infty}I_2^{(\gamma)}(u_0^n)= J_2^{(\gamma)}$. Thus
\begin{equation}\label{difference limit}
  K_{\gamma}(u^n(t^n))= K_{\gamma}(u_0^n) \to 0, \qquad  \mathrm{as}\qquad n \to +\infty.
\end{equation}In order to construct a minimizing sequence of $I_2^{(\gamma)}$, we need to prove the following inequalities
\begin{equation}\label{estimate of liminf limsup of L6 norm of un(tn)}
C(\gamma)^{-1} b \leq \liminf_{n\to +\infty}\|u^n(t^n)\|_{L^6}\leq \limsup_{n\to +\infty}\|u^n(t^n)\|_{L^6} \leq C(\gamma)b.
\end{equation}In fact, we denote by $C_1^6:= \tfrac{1+\gamma}{J_2^{(\gamma)}} C(\gamma)^6 $, then Sobolev embedding $\|f\|_{L^6} \leq C_1 \||D|^{\frac{1}{3}} f\|_{L^2}$ yields that
\begin{equation*}
\frac{\|\partial_x u^n(t^n)\|_{L^2}^2}{2}+  \frac{\gamma\||D|^{\frac{1}{2}} u^n(t^n)\|_{L^2}^4}{2\|u^n(t^n)\|_{L^2}^2} \geq  \frac{(1+\gamma)\||D|^{\frac{1}{2}} u^n_0\|_{L^2}^4}{2\|u^n_0\|_{L^2}^2} \geq \frac{(1+\gamma)\||D|^{\frac{1}{3}} u^n_0\|_{L^2}^6}{2\|u^n_0\|_{L^2}^4} \geq \frac{(1+\gamma)\|  u^n_0\|_{L^6}^6}{2 C_1^6\|u^n_0\|_{L^2}^4}.
\end{equation*}Thus $\|u^n(t^n)\|_{L^6} \geq \sqrt[6]{\frac{3(1+\gamma)\|  u^n_0\|_{L^6}^6}{ C_1^6\|u^n_0\|_{L^2}^4} - 6 K_{\gamma}(u^n(t^n))} \to \frac{b}{C(\gamma)}$, as $n\to +\infty$. Similarly, we have
\begin{equation*}
\|u^n(t^n)\|_{L^6} \leq C_1 \||D|^{\frac{1}{3}}u^n(t^n)\|_{L^2} \leq C_1 \||D|^{\frac{1}{2}}u^n(t^n)\|_{L^2}^{\frac{2}{3}}\| u^n(t^n)\|_{L^2}^{\frac{1}{3}}=C_1 \||D|^{\frac{1}{2}}u^n_0\|_{L^2}^{\frac{2}{3}}\| u^n_0\|_{L^2}^{\frac{1}{3}},
\end{equation*}and $\|u^n(t^n)\|_{L^6} \leq C_1\|\partial_x u^n_0\|_{L^2}^{\frac{1}{3}}\| u^n_0\|_{L^2}^{\frac{2}{3}}$. Thus we have $\|u^n(t^n)\|_{L^6} \leq C_1\|u^n_0\|_{L^6} \sqrt[6]{\frac{I_2^{(\gamma)}(u^n_0)}{1+\gamma}}\to C(\gamma) b$, as $n\to +\infty$. Thus estimates $(\ref{estimate of liminf limsup of L6 norm of un(tn)})$ are proved and we have $\lim_{n\to +\infty} I_2^{(\gamma)}(u^n(t^n)) =J_{2}^{(\gamma)}$.\\

\noindent Rescaling $v_n(x):=\lambda_n u^n(t^n, \mu_nx)$ such that $\|v_n\|_{L^2}=\|v_n\|_{L^6}=1$, with $\lambda_n, \mu_n >0$. Then 
\begin{equation}
\sqrt{\mu_n} \lambda_n^{-1} = \|u_0^n\|_{L^2}, \qquad \sqrt[6]{\mu_n} \lambda_n^{-1} = \|u^n(t^n)\|_{L^6}, \qquad \lim_{n\to +\infty} I_2^{(\gamma)}(v_n) =J_{2}^{(\gamma)}
\end{equation}

\begin{equation*}
\|\partial_x v_n\|_{L^2}^2= I(v_n)=I(u^n(t^n))=\frac{\|\partial_x u^n(t^n)\|_{L^2}^2}{\|u^n(t^n)\|_{L^6}^6} \| u^n_0\|_{L^2}^4 \to J, \qquad \mathrm{as} \qquad n \to +\infty.
\end{equation*}Theorem $\ref{existence of minimizer of I m gamma general}$ yields that there exists a profile $U\in G_2^{(\gamma)}(1,1)$ and a sequence of real numbers $(x_n)_{n\in \mathbb{N}}$ such that $\|v_n-U(\cdot -x_n)\|_{H^1} \to 0$, as $n \to +\infty$ up to a subsequence, still denoted by $(v_n)_{n\in \mathbb{N}_+}$. Moreover, we assume that $\|u^n(t^n)\|_{L^6} \to \theta b$, as $n \to +\infty$ in the same subsequence. Then $(\ref{estimate of liminf limsup of L6 norm of un(tn)})$ yields that $C^{-1}\leq \theta \leq C$. We denote by $\lambda_{\infty}=\lim_{n \to +\infty}\lambda_n$ and $\mu_{\infty}=\lim_{n \to +\infty}\mu_n$. Then we have
\begin{equation}\label{lambda + infinity mu infinity valor}
\lambda_{\infty} >0, \qquad \mu_{\infty} >0, \qquad \sqrt{\mu_{\infty}} \lambda_{\infty}^{-1}=\sqrt[4]{3J_2^{(\gamma)}}\qquad \mathrm{and} \qquad \sqrt[6]{\mu_{\infty}} \lambda_{\infty}^{-1} =\theta b.
\end{equation}Then
\begin{equation}\label{auxilary for estimate of difference between V and un tn}
\lim_{n\to \infty}\|\lambda_n u^n(t^n, \mu_n \cdot)-U(\cdot - x_n)\|_{H^1}=0 \quad \Longleftrightarrow \quad\lim_{n\to \infty}\|u^n(t^n)- \frac{1}{\lambda_n} U(\frac{\cdot-\mu_n x_n}{\mu_n})\|_{H^1} =0.
\end{equation}Since $U\in H^1_+$, we have $\lim_{n\to \infty}\|\frac{1}{\lambda_n}U(\frac{\cdot}{\mu_n})-\frac{1}{\lambda_{\infty}}U(\frac{\cdot}{\mu_{\infty}})\|_{H^1}= 0$. Together with $(\ref{auxilary for estimate of difference between V and un tn})$, we have
\begin{equation}
\|u^n(t^n) - \Psi(\cdot-\mu_n x_n)\|_{H^1}\to 0, \qquad \mathrm{as} \qquad n \to +\infty,
\end{equation}where $\Psi(x):= \frac{1}{\lambda_{\infty}}U(\frac{x}{\mu_{\infty}})$. Since $\|U\|_{L^2}=\|U\|_{L^6}=1$, we have $\Psi  \in G_2^{(\gamma)}(\sqrt[4]{3J_2^{(\gamma)}}, \theta b)$ by $(\ref{lambda + infinity mu infinity valor})$, leading to $(\ref{goal of orbital stability alpha =0})$ up to a subsequence. 
\end{proof}

\bigskip

\subsection{The special case $\gamma=2$}
\noindent We prove proposition $\ref{Cauchy Schwarz estimate to find minimizer}$ in order to classify all the ground states of the functional
\begin{equation*}
I_2^{(2)}=\frac{\|\partial_x f\|_{L^2}^2\|f\|_{L^2}^4+2\||D|^{\frac{1}{2}}f\|_{L^2}^4\|f\|_{L^2}^2}{\|f\|_{L^6}^6}, \qquad \forall f\in H^1_+ \backslash \{0\}.
\end{equation*}as $G_2^{(2)} = \{x\mapsto\frac{\lambda e^{i\theta}}{\mu x+ y +i} \in H^1_+: \forall \lambda, \mu >0 \quad \forall \theta, y \in \mathbb{R}\}$. The idea of using Cauchy--Schwarz inequality to classify ground states follows from Foschi $[\ref{Foschi, Maximizers for the Strichartz Inequality}]$ for linear Schr\"odinger equation and linear wave equation on $\mathbb{R}^d$, for $d \geq 1$, G\'erard--Grellier $[\ref{gerardgrellier1}, \ref{Gerard grellier book cubic szego equation and hankel operators}]$ for the cubic Szeg\H{o} equation on the torus $\mathbb{S}^1$, Pocovnicu $[\ref{pocovnicu Traveling waves for the cubic Szego eq}]$ for the cubic Szeg\H{o} equation on the line $\mathbb{R}$ and Gassot $[\ref{Gassot radially symmetric traveling waves}]$ for the non linear Schr\"odinger equation on the Heisenberg group.\\

\begin{proof}[Proof of proposition $\ref{Cauchy Schwarz estimate to find minimizer}$]
Plancherel formula gives that $\|f\|_{L^6}^6 = \frac{1}{32\pi^5}\int_{\xi >0} |(\hat{f}*\hat{f}*\hat{f})(\xi)|^2 \mathrm{d}\xi$. By using Cauchy--Schwarz inequality, for every $\xi >0$, we have
\begin{equation*}
\begin{split}
|(\hat{f}*\hat{f}*\hat{f})(\xi)|^2 = & \Big|\int_{\eta_1 >0, \eta_2>0, \eta_1+\eta_2\leq \xi}\hat{f}(\eta_1)\hat{f}(\eta_2)\hat{f}(\xi-\eta_1-\eta_2)\mathrm{d}\eta_1 \mathrm{d}\eta_2\Big|^2\\
\leq & \frac{\xi^2}{2}|\int_{\eta_1 >0, \eta_2>0, \eta_1+\eta_2\leq \xi}\Big|\hat{f}(\eta_1)\hat{f}(\eta_2)\hat{f}(\xi-\eta_1-\eta_2)\Big|^2\mathrm{d}\eta_1 \mathrm{d}\eta_2.\\
\end{split}
\end{equation*}Thus we have
\begin{equation*}
\begin{split}
\|f\|_{L^6}^6 \leq &\frac{1}{64\pi^5} \int_{\eta_1 >0, \eta_2>0, \eta_3>0}(\eta_1+\eta_2+\eta_3)^2\Big|\hat{f}(\eta_1)\hat{f}(\eta_2)\hat{f}(\eta_3)\Big|^2\mathrm{d}\eta_1 \mathrm{d}\eta_2 \mathrm{d}\eta_3\\
= & \frac{3}{8\pi^2}(\|\partial_x f\|_{L^2}^2 \| f\|_{L^2}^4 + 2 \||D|^{\frac{1}{2}} f\|_{L^2}^4\| f\|_{L^2}^2).
\end{split}
\end{equation*}If $I_2^{(2)}(f)=\frac{8 \pi^2}{3}$, then $\hat{f}(\eta_1)\hat{f}(\eta_2)\hat{f}(\eta_3)=\hat{f}(0)^2 \hat{f}(\eta_1+\eta_2+\eta_3)$, for all $\eta_1, \eta_2,\eta_3 >0$. Since $f\in H^1_+\backslash \{0\}$, we have $\hat{f}(0) \ne 0$. Thus 
\begin{equation*}
\hat{f}(\eta_1)\hat{f}(\eta_2)=\hat{f}(\eta_1+\eta_2)\hat{f}(0), \qquad \forall \eta_1, \eta_2 \geq 0.
\end{equation*}This is true if and only if $\hat{f}(\eta)=e^{-ip \eta} \hat{f}(0)$, for some $p \in \mathbb{C}$ such that $\mathrm{Im}p<0$. Thus we have $f(x)=\frac{\mathcal{A}}{x-p}$, for some $\mathcal{A} \in \mathbb{C}$. We conclude by $I_2^{(2)}(f)=I_2^{(2)}(f_{\lambda,  \mu, \theta, y})$, with $f_{\lambda,  \mu, \theta, y}(x)=\lambda e^{i\theta} f(\mu x+y)$.
\end{proof}

 \bigskip

\noindent For every $c>0$, we prove the orbital stability of traveling wave $u(t,x)=e^{\frac{3c^2 it}{8}}Q_c^*(x+ct)$ of equation $(\ref{NLS-Szego quintic R})$, with $Q_c^*(x)= \frac{2\sqrt[4]{2c^2}}{cx+2i}$. Proposition $\ref{Cauchy Schwarz estimate to find minimizer}$ yields that $G_2^{(2)}(\sqrt[4]{8\pi^2}, \sqrt[6]{\frac{3\pi c^2}{\sqrt{2}}})=\{x\mapsto \frac{2\sqrt[4]{2c^2}e^{i\theta}}{cx+2i+y} \in H^1_+ : \forall \theta, y \in \mathbb{R}\}$. If $Q_c \in G_2^{(2)}(\sqrt[4]{8\pi^2}, \sqrt[6]{\frac{3\pi c^2}{\sqrt{2}}})$, then we have 
\begin{equation}\label{norm classification of G 2 2}
\|Q_c\|_{L^2}^4=8\pi^2, \qquad \|\partial_x Q_c\|_{L^2}^2 =\frac{\pi c^2}{2 \sqrt{2}}, \qquad \||D|^{\frac{1}{2}}Q_c\|_{L^2}^2 = \frac{\pi c}{\sqrt{2}},\qquad \|Q_c\|_{L^6}^6=\frac{3\pi c^2}{\sqrt{2}}
\end{equation}

\begin{proof}[Proof of theorem $\ref{Orbital stability of ground state of I m=2 gamma=2}$]
Fix $c>0$, for every $n \in \mathbb{N}$, we choose $u_0^n\in H^1_+$ and $Q_c^n \in G_2^{(2)}(\sqrt[4]{8\pi^2}, \sqrt[6]{\frac{3\pi c^2}{\sqrt{2}}})$ such that $\|u_0^n-Q_c^n\|_{H^1} \to 0$, as $n \to +\infty$. Let $u^n$ solve $(\ref{NLS-Szego quintic R})$ with initial datum $u^n(0)=u^n_0$. We shall prove that
\begin{equation}\label{goal of orbital stability alpha =0 tilde}
\inf_{\Psi \in G_2^{(2)}(\sqrt[4]{8\pi^2}, \sqrt[6]{\frac{3\pi c^2}{\sqrt{2}}})}\|u^n(t^n) - \Psi\|_{H^1} \to 0,  \qquad \mathrm{as} \qquad n \to +\infty,
\end{equation}up to a subsequence, for every temporal sequence $(t_n)_{n\in \mathbb{N}} \subset \mathbb{R}$. We use the same procedure as the proof of theorem $\ref{Orbital stability of ground state of I m=2 gamma positif}$ to obtain that $(u^n(t^n))_{n\in \mathbb{N}}$ is a minimizing sequence of $I_2^{(2)}$. We set $v_n(x):=\lambda_n u^n(t^n, \mu_nx)$ such that $\|v_n\|_{L^2}=\|v_n\|_{L^6}=1$, with $\lambda_n, \mu_n >0$ and $\lim_{n \to +\infty}I_2^{(2)}(v_n)=\lim_{n \to +\infty}I_2^{(2)}(u^n(t^n))=J_2^{(2)}$.
\begin{equation*}
\tilde{I}(v_n)=\tilde{I}(u^n(t^n))\to 1, \qquad \mathrm{as} \qquad n \to +\infty.
\end{equation*}Theorem $\ref{existence of minimizer of I m gamma general}$ yields that there exists a profile $V\in G_2^{(2)}(1,1)$ and a sequence of real numbers $(y_n)_{n\in \mathbb{N}_+}$ such that $\|v_n-V(\cdot -y_n)\|_{H^1} \to 0$, as $n \to +\infty$ up to a subsequence, still denoted by $(v_n)_{n\in \mathbb{N}_+}$. Similarly, we have $(\ref{auxilary for estimate of difference between V and un tn})$ for $V$.
\begin{equation}\label{auxilary for estimate of difference between V and un tn  tilde}
\lim_{n\to \infty}\|\lambda_n u^n(t^n, \mu_n \cdot)-V(\cdot - y_n)\|_{H^1}=0 \quad \Longleftrightarrow \quad\lim_{n\to \infty}\|u^n(t^n)- \frac{1}{\lambda_n} V(\frac{\cdot-\mu_n y_n}{\mu_n})\|_{H^1} =0.
\end{equation}Since all the ground states are completely classified by proposition $\ref{Cauchy Schwarz estimate to find minimizer}$, we obtain the values of $\lambda_{\infty}$ and $\mu_{\infty}$ from $(\ref{norm classification of G 2 2})$ and the conservation law $P(u)=\||D|^{\frac{1}{2}} u\|_{L^2}^2$. Since $\lambda_n^2 \||D|^{\frac{1}{2}} \frac{1}{\lambda_n} V(\frac{\cdot-\mu_n y_n}{\mu_n})\|_{L^2}^2 =\sqrt{\frac{2}{3}}\pi$, $\||D|^{\frac{1}{2}} u^n(t^n)\|_{L^2}^2=\||D|^{\frac{1}{2}} u^n_0\|_{L^2}^2 \to \frac{\pi c}{\sqrt{2}}$ and
\begin{equation*}
\Big|\||D|^{\frac{1}{2}} u^n(t^n)\|_{L^2}-\||D|^{\frac{1}{2}} \frac{1}{\lambda_n} V(\frac{\cdot-\mu_n y_n}{\mu_n})\|_{L^2}\Big| \to 0, \qquad \mathrm{as}\qquad n\to +\infty,
\end{equation*}we have $\lambda_{\infty}^2:=\lim_{n \to +\infty}\lambda_n^2 = \frac{2}{\sqrt{3}c}$ and $\mu_{\infty}:=\lim_{n \to +\infty}\mu_n = \sqrt{8 \pi^2}\lim_{n \to +\infty}\lambda_n^2 =\frac{4\sqrt{2}\pi}{\sqrt{3}c}$. Together with $(\ref{auxilary for estimate of difference between V and un tn tilde})$, we have
\begin{equation*}
\|u^n(t^n) - \tilde{\Psi}(\cdot-\mu_n y_n)\|_{H^1}\to 0, \qquad \mathrm{as} \qquad n \to +\infty,
\end{equation*}where $\tilde{\Psi}(x):= \frac{1}{\lambda_{\infty}}V(\frac{x}{\mu_{\infty}})$. We have $\tilde{\Psi}  \in G_2^{(2)}(\sqrt[4]{8\pi^2}, \sqrt[6]{\frac{3\pi c^2}{\sqrt{2}}})$, leading to $(\ref{goal of orbital stability alpha =0 tilde})$ up to a subsequence. 
\end{proof}

\bigskip
\bigskip

\section{Appendices}
\noindent In the first appendix, we prove that if a $H^1-$solution of equation $(\ref{NLS-Szego quintic R})$ $L^2-$scatters, then it also $H^1-$scatters. Then we discuss the problem of uniqueness of ground states of the functional $I_2^{(\gamma)}$, for general $\gamma \geq 0$.
\subsection{Persistence of regularity for scattering}
\noindent The persistence of regularity for scattering can be established by Strichartz estimates and a bootstrap argument.
\begin{prop}\label{persistence of regularity of scattering}
If $u_0 \in H^1_+$ and there exists $u_+\in L^2_+$ such that $\lim_{t\to+\infty}\|u(t)-e^{it\partial_x^2} u_+\|_{L^2}=0$, where $u$ is the unique solution of $(\ref{NLS-Szego quintic R})$, then we have $u_+=u_0 + i \int_0^{+\infty}e^{-i\tau \partial_x^2} \Pi(|u(\tau)|^4 u(\tau)) \mathrm{d}\tau \in H^1_+$ and $\lim_{t\to+\infty}\|u(t)-e^{it\partial_x^2} u_+\|_{H^1}=0$.
\end{prop}

\begin{proof}
We claim that $u\in L^6(0,+\infty; L^6( \mathbb{R}))$. In fact $(6,6)$ is $1$-admissible. Set $v(t,x):=e^{it\partial_x^2} u_+(x)$ and $r(t,x)=u(t,x)-v(t,x)$, then we have $v  \in L^6(\mathbb{R}_t \times \mathbb{R}_x)$ by Strichartz inequality and 
\begin{equation}\label{pde of r=u-v}
i\partial_t r +\partial_x^2 r = -\Pi(|u|^4 u).
\end{equation}
\begin{equation}\label{Duhamel formula for r}
r(t)= e^{i(t-\sigma)\partial_x^2} r(\sigma) + i\int_{\sigma}^{t} e^{i(t-\tau)\partial_x^2}\Pi(|u(\tau)|^4 u(\tau))\mathrm{d}\tau.
\end{equation}Since $u\in L^{\infty}(\mathbb{R}; H^1_+)$, we have $r\in L^6_{\mathrm{loc}}(\mathbb{R}_+, L^6(\mathbb{R}))$. Recall that $\Pi$ is bounded $L^p \to L^p$, for all $1<p<+\infty$. Applying Strichartz inequality to formula $(\ref{Duhamel formula for r})$, then there exists a constant $C^*>0$ such that
\begin{equation*}
\begin{split}
\|r\|_{L^6(T, T'; L^6(\mathbb{R}))}\leq & C^* \left( \|r(T)\|_{L^2} + \||u|^4 u\|_{L^{\frac{6}{5}}(T, T'; L^{\frac{6}{5}}(\mathbb{R}))}\right)\\
 \leq & C^*\left(\|r(T)\|_{L^2} + \sum_{k=0}^5 \|v\|_{L^6(T, T'; L^6(\mathbb{R}))}^k \|r\|_{L^6(T, T'; L^6(\mathbb{R}))}^{5-k}\right)
 \end{split}
\end{equation*}holds for all $0\leq T <T'$. Set $\delta=\min\{1, \frac{1}{2C^*(2+C^*)^4}\}$, there exists $T>0$ such that $\|v\|_{L^6(T,+\infty;L^6(\mathbb{R}))} \leq \delta$ and $\|r(T)\|_{L^2}  \leq \delta$. We choose $T_1:=\sup\{S \in [T,+\infty): \|r\|_{L^6(T,S;L^6(\mathbb{R}))} \leq \frac{1}{2+C^*}\}$. For every $T'\in (T,T_1)$, we have $\|r\|_{L^6(T, T'; L^6(\mathbb{R}))}\leq C^* \left(\delta + \delta^5 + (2+C^*)^{-5} + 4\delta \|r\|_{L^6(T,T';L^6(\mathbb{R}))}\right)$. Then we have
\begin{equation*}
\|r\|_{L^6(T, T'; L^6(\mathbb{R}))}\leq 2\delta C^*+ \frac{1}{(2+C^*)^4} \leq \frac{1}{(2+C^*)^3} < \frac{1}{2+C^*}.
\end{equation*}Thus $T_1 = +\infty$ and $\|r\|_{L^6(T, +\infty; L^6(\mathbb{R}))}< \frac{1}{2+C^*}$, which yields that $u \in L^6(0,+\infty, L^6(\mathbb{R}))$.\\

\noindent Set $\tilde{u}_+:= u_0 + i \int_0^{+\infty}e^{-i\tau \partial_x^2} \Pi(|u(\tau)|^4 u(\tau)) \mathrm{d}\tau$. Since $u \in L^6(\mathbb{R}_t \times \mathbb{R}_x)$, Strichartz inequality implies that
\begin{equation}
\|u(t)-e^{it\partial_x^2} \tilde{u}_+\|_{L^2}\lesssim \|\Pi(|u|^4 u)\|_{L^{\frac{6}{5}}(t,+\infty; L^{\frac{6}{5}}(\mathbb{R}))} \lesssim \|u\|_{L^6(t,+\infty; L^6(\mathbb{R}))}^5 \to 0, \qquad \mathrm{as} \quad t \to +\infty.
\end{equation}Thus $\tilde{u}_+=u_+$. Since the momentum $P(u)=\langle -i\partial_x u, u\rangle_{L^2}=\||D|^{\frac{1}{2}}u\|_{L^2}^2$ is conserved, we have 
\begin{equation}
\sup_{t\in \mathbb{R}} \|\partial_x u(t)\|_{L^2}^2 \lesssim E(u_0)+ \||D|^{\frac{1}{2}}u_0\|_{L^2}^{4} \|u_0\|_{L^2}^{2}.
\end{equation}Thus $\partial_x \circ \Pi(|u|^4 u) \in L^{\infty}(\mathbb{R}; L^2_+) \hookrightarrow L^1_{loc}(\mathbb{R}; L^2_+)$. The Strichartz estimate yields that $\forall T>0$, we have
\begin{equation}
\|\partial_x u\|_{L^6(0, T; L^6(\mathbb{R}))} \lesssim \|\partial_x u_0\|_{L^2} + \int_0^T \|\partial_x \Pi(|u(\tau)|^4 u(\tau))\|_{L^2} \mathrm{d}\tau<+\infty.
\end{equation}Since $u \in L^6(\mathbb{R}_t \times \mathbb{R}_x)$, there exists $T_0>0$ such that $\|u\|^4_{L^6(T_0,+\infty; L^6(\mathbb{R}))} \leq \frac{1}{2C^*}$. For all $T>T_0$, we have
\begin{equation*}
\begin{split}
\|\partial_x u\|_{L^6(T_0,T;L^6(\mathbb{R}))} \leq & C^*(\|\partial_x u_0\|_{L^2}+ \||\partial_x u| |u|^4\|_{L^{\frac{6}{5}}(T_0,T; L^{\frac{6}{5}}(\mathbb{R}))})\\
\leq & C^*(\|\partial_x u_0\|_{L^2}+ \| u\|^4_{L^{6}(T_0,+\infty; L^{6}(\mathbb{R}))}\|\partial_x u\|_{L^{6}(T_0,T; L^{6}(\mathbb{R}))})\\
\end{split}
\end{equation*}Thus $\|\partial_x u\|_{L^6(T_0,+\infty;L^6(\mathbb{R}))} \leq 2C^* \|\partial_x u_0\|_{L^2}$ and $\partial_x u \in L^6(0,+\infty;L^6(\mathbb{R}))$. Consequently, we use Strichartz estimate to obain
\begin{equation*}
\|\partial_x( u(t)-e^{it\partial_x^2} u_+)\|_{L^2}  \lesssim \|u\|_{L^6(t,+\infty; L^6(\mathbb{R}))}^4  \|\partial_x u\|_{L^6(t,+\infty; L^6(\mathbb{R}))}\to 0,
\end{equation*}as $t \to +\infty$.
\end{proof}

\bigskip

\noindent Similarly, we have the following proposition for scattering backward in time.

\begin{prop}\label{persistence of regularity of scattering negative time}
If $u_0 \in H^1_+$ and there exists $u_-\in L^2_+$ such that $\lim_{t\to-\infty}\|u(t)-e^{it\partial_x^2} u_-\|_{L^2}=0$, where $u$ is the unique solution of $(\ref{NLS-Szego quintic R})$, then we have $u_-=u_0 - i \int_{-\infty}^0 e^{-i\tau \partial_x^2} \Pi(|u(\tau)|^4 u(\tau)) \mathrm{d}\tau \in H^1_+$ and $\lim_{t\to+\infty}\|u(t)-e^{it\partial_x^2} u_-\|_{H^1}=0$.
\end{prop}

\bigskip
 
\noindent The $L^r-$norm of a solution of linear Sch\"odinger equation decays as $t \to \pm\infty$, for all $2< r \leq +\infty$.

\begin{lem}\label{Lr estimates to 0 when time go to infinity}
If $f \in H^1(\mathbb{R})$ and $2<r \leq +\infty$, then $\|e^{it\partial_x^2} f\|_{L^r} \to 0$, as $|t| \to +\infty$.
\end{lem}

\begin{proof}
We set $w(t):=e^{it\partial_x^2} f$ and $q=\frac{4r}{r-2}$, then $\frac{2}{q}+\frac{1}{r}=\frac{1}{2}$ and $w\in L^{\infty}(\mathbb{R},H^1_+) \bigcap L^q(\mathbb{R};L^r(\mathbb{R}))$ by Strichartz estimate. Gagliardo--Nirenberg inequality yields that
\begin{equation*}
\|w(t)-w(s)\|_{L^r} \leq \|\partial_x w(t) - \partial_x w(s)\|_{L^2}^{\frac{2}{q}} \|w(t)-w(s)\|_{L^2}^{1-\frac{2}{q}}\leq 2 \|\partial_x f\|_{L^2}^{\frac{2}{q}} \|w(t)-w(s)\|_{L^2}^{1-\frac{2}{q}}.
\end{equation*}Since $\sup_{t \in \mathbb{R}}\|\partial_t w(t) \|_{H^{-1}(\mathbb{R})} = \sup_{t \in \mathbb{R}}\|\partial_x^2 w(t) \|_{H^{-1}(\mathbb{R})}\leq \|f\|_{H^1}$, the mapping $t \to w(t)$ is Lipschitz continuous $\mathbb{R} \to H^{-1}(\mathbb{R})$. Therefore, 
\begin{equation*}
\|w(t)-w(s)\|_{L^2} \leq \sqrt{\|w(t)-w(s)\|_{H^1}\|w(t)-w(s)\|_{H^{-1}}} \lesssim \|f\|_{H^1} |t-s|^{\frac{1}{2}}
\end{equation*}and $\|w(t)-w(s)\|_{L^r}\lesssim \|f\|_{H^1} |t-s|^{\frac{q-2}{2q}}$. The mapping $t \to w(t)$ is uniformly continuous $\mathbb{R} \to L^r(\mathbb{R})$. Since $w \in L^q(\mathbb{R};L^r(\mathbb{R}))$, we have $\|e^{it\partial_x^2} f\|_{L^r} = \|w(t)\|_{L^r} \to 0$, as $|t| \to +\infty$. 
\end{proof}

\bigskip

\noindent This lemma yields that a traveling wave $u(t,x)=e^{i\omega t}Q(x)$ does not $H^1-$scatter neither $L^2-$scatter, with $\omega>0$ and $Q^{(0)} \in G_2^{(0)}(\sqrt[4]{3J_2^{(0)}}, (\frac{3\sqrt{3J_2^{(0)}}\omega}{2})^{\frac{1}{6}})$. Together with theorem $\ref{Orbital stability of ground state of I m=2 gamma positif}$, we can prove corollary $\ref{scattering threshold reduced}$.

\bigskip

\subsection{Open problem of uniqueness of ground states}\label{subsection of Problem of uniqueness of ground states for general gamma}

\noindent The problem of classification of ground states of $I_2^{(\gamma)}$ remains open for general $\gamma$, since it is difficult to solve the non local equation $(\ref{equation for bounded state Q})$. However, for every $m\in \mathbb{N}$, if $f$ is a ground state of $I_m^{(\gamma)}$, then so is  $P(f)$, where $P(f)(x)= \frac{1}{2\pi}\int_{0}^{+\infty}|\hat{f}(\xi)|e^{ix \xi}\mathrm{d}\xi$. Precisely, we have the following proposition.
\begin{prop}\label{fourier transformation positive}
For every $m \in \mathbb{N}$ and $\gamma\geq 0$, if $f \in G_m^{(\gamma)}$, then $P(f)\in G_m^{(\gamma)}$ and there exist $a, b \in \mathbb{R}$ such that $\hat{f}(\xi)=|\hat{f}(\xi)|e^{i(ax+b)}$, for every $\xi \geq 0$.
\end{prop}
\begin{proof}Parseval identity implies that
\begin{equation*}
\||D|^s P(f)\|_{L^2}^2 =\frac{1}{2\pi}\int_{0}^{+\infty}|\xi|^{2s}|\hat{f}(\xi)|^2  \mathrm{d}\xi= \||D|^s f\|_{L^2}^2, \qquad \forall 0\leq s \leq 1.
\end{equation*}However, since $\widehat{P(f)}(\xi)=|f(\xi)|$, for every $\xi >0$,  we have $\widehat{P(f)^{m+1}} \geq |\widehat{f^{m+1}}|$, for every $m\in \mathbb{N}$ and
\begin{equation*}
\|P(f)\|_{L^{2m+2}}^{2m+2} = \frac{1}{2 \pi }\int_{0}^{+\infty}| \widehat{P(f)^{m+1}}(\xi)|^2\mathrm{d}\xi  \geq \frac{1}{2 \pi }\int_{0}^{+\infty}| \widehat{f^{m+1}}(\xi)|^2\mathrm{d}\xi = \|f\|_{L^{2m+2}}^{2m+2}
\end{equation*}Thus $J_m^{(\gamma)} \leq I_m^{(\gamma)}(P(f)) \leq  I_m^{(\gamma)}(f) = J_m^{(\gamma)}$, for all $\gamma \geq 0$ and $m\in \mathbb{N}$. So $P(f) \in G_m^{(\gamma)}$ and all precedent inequalities become equalities. In particular,
$\widehat{P(f)^{m+1}} = |\widehat{f^{m+1}}|$. We set $h(\xi)=\hat{f}(\xi)$,
then the Euler--Lagrange equation $(\ref{Euler Lagrange equation of I m gamma})$ reads in Fourier modes as
\begin{equation}\label{Euler--Lagrange equation of I m gamma fourier mode}
(A_m(f)+B_{m,\gamma}(f)\xi + C_m(f)\xi^2)h(\xi)=D_{m,\gamma} \mathds{1}_{\xi\geq 0}T(h,h,\cdots, h)(\xi),
\end{equation}where $A_m(f), B_{m,\gamma}(f),  C_m(f), D_{m,\gamma}>0$ and $T : L^1(\mathbb{R}_+)^{2m+1} \to L^1(\mathbb{R}_+) $ is $(2m+1)$-linear defined as
\begin{equation*}
T(h_1,h_2,\cdots,h_{2m+1})(\xi)=\int_{\mathcal{S(\xi)}}\Pi_{j=1}^{2m+1}h_j(\eta_j) \mathrm{d}\eta_1\cdots\mathrm{d}\eta_{2m+1} 
\end{equation*}with $\mathcal{S(\xi)}=\{(\eta_1, \cdots, \eta_{2m+1}) \in \mathbb{R}_+^{2m+1}: \sum_{j=1}^{m+1}\eta_j = \sum_{j=m+2}^{2m+1}\eta_k +\xi,\quad \mathrm{and}\quad  \eta_j, \zeta_k \geq 0 \}$. \\

\noindent We claim that if $h(\xi)=0$ for some $\xi \in \mathbb{R}_+$ then $h\equiv 0$. In fact, assume by contradiction that $h(\zeta)\ne 0$, for some $\zeta \geq 0$. Since $P(f) \in G_m^{(\gamma)}$, we replace $h$ by $|h|$ in equation $(\ref{Euler--Lagrange equation of I m gamma fourier mode})$ in order to get the following implication:
\begin{equation*}
h(\xi) =0 \Longrightarrow  h(\frac{m \zeta +\xi}{m+1})=0.
\end{equation*}We construct an iterative sequence $\xi_0=\xi$ and $\xi_{n+1}=\frac{m \zeta +\xi_n}{m+1}$, $\forall n \in \mathbb{N}$. Then we have $h(\xi_n)=0$ and $\lim_{n\to +\infty}\xi_n = \zeta$. The continuity of $h$ gives that $h(\zeta)=0$, contradiction. Thus $h\equiv 0$.\\

\noindent Since $h=\hat{f}$ and $f \ne 0$, then $h$ is continuous $\mathbb{R} \to \mathbb{C}^*$. Thus there exists a continuous function $\alpha : \mathbb{R}_+ \to \mathbb{S}^1$ such that
\begin{equation*}
\Pi_{j=1}^{m+1} h(\xi_j) =\alpha(\sum_{j=1}^{m+1} \xi_j) \Pi_{j=1}^{m+1} |h(\xi_j)|.
\end{equation*}The lifting theorem yields that there exists a unique continuous function $\varphi : \mathbb{R}_+ \to \mathbb{R}$ such that 
\begin{equation*}
h(\xi)= |h(\xi)| e^{i\varphi(\xi)}, \qquad \sum_{j=1}^{m+1}\varphi(\xi_j)=\beta(\sum_{j=1}^{m+1} \xi_j)
\end{equation*}for some continuous function $\beta : \mathbb{R}_+ \to \mathbb{R}$. We set $\phi(\xi)=\varphi(\xi)-\varphi(0)$ then we have 
\begin{equation*}
\phi(\xi_1) + \phi(\xi_2) = 2 \phi(\frac{\xi_1 + \xi_2}{2}), \qquad \forall \xi_1, \xi_2 \geq 0.
\end{equation*}Consequently, $\varphi(\xi)=\phi(1)\xi+\varphi(0)=(\varphi(1)-\varphi(0))\xi+\varphi(0)$, for every $\xi\geq 0$.
\end{proof}

\noindent Thus it suffices to study the uniqueness of ground states modulo the positive Fourier transformation.\\

\noindent We compare theorem $\ref{Orbital stability of ground state of I m=2 gamma positif}$ and theorem $\ref{Orbital stability of ground state of I m=2 gamma=2}$. Since we do not know the uniqueness of ground states of the functional $I_2^{(\gamma)}$, the conservation law $P(u)=\||D|^{\frac{1}{2}}u\|_{L^2}^2$ can not be completely used to determine the $L^6-$norm of the final profile that approaches $u(t)$. However, if the ground state is unique up to scaling, phase rotation and spatial translation, then we can determine the $L^6-$norm of the profile that approaches $u(t)$. So we have the actual orbital stability in the case $\gamma=2$.


\bigskip
\bigskip







\bigskip
\bigskip






\begin{thebibliography}{99}

\bibitem{Cazenave introduction to NLS Rio de Janeiro}Cazenave, T.
\emph{An introduction to nonlinear Schr\"odinger equations}, Textos de M\'etodos Matem\'aticos $\mathbf{26}$, Instituto de Matem\'atica UFRJ, 1996  \label{Cazenave introduction to NLS Rio de Janeiro} 

\bibitem{Cazenave book NLS AMS}Cazenave, T.
\emph{Semilinear Schr\"odinger equations}, Courant Lecture Notes in Mathematics, $\mathbf{10}$. New York University, Courant Institute of Mathematical Sciences, AMS, (2003)  \label{Cazenave book NLS AMS} 

\bibitem{Cazenave-lions Orbital stability of standing waves }Cazenave, T., Lions, P.-L.
\emph{Orbital stability of standing waves for some nonlinear Schr\"odinger equations}, Commun. Math. Phys. $85$ (1982), 549—561  \label{Cazenave-lions Orbital stability of standing waves } 


\bibitem{Cazenave-weissler1}Cazenave, T., Weissler, F.
\emph{The Cauchy problem for the nonlinear Schr\"odinger equation in $H^1$}, Manuscripta. Mathematics, $61$ (1988), 477-494  \label{Cazenave-weissler1} 

\bibitem{Cazenave-weissler2}Cazenave, T., Weissler, F.
\emph{The Cauchy problem for the nonlinear Schr\"odinger equation in $H^s$},  Nonlinear Analysis, $14$ (1990), 807-836  \label{Cazenave-weissler2} 

\bibitem{Dodson, Global well-posedness and scattering }Dodson, B.
\emph{Global well-posedness and scattering for the mass critical nonlinear Schr\"odinger equation with mass below the mass of the ground state}, Adv. Math., 285(2015), 1589-1618 \label{Dodson, Global well-posedness and scattering } 



\bibitem{Foschi, Maximizers for the Strichartz Inequality}Foschi, D.
\emph{Maximizers for the Strichartz Inequality}, J. Eur. Math. Soc. (JEMS) 9 (2007), no. 4, 739–774. \label{Foschi, Maximizers for the Strichartz Inequality} 



\bibitem{Frank Lenzmann Uniqueness of non linear ground states}Frank, R., Lenzmann, E., \emph{Uniqueness of non-linear ground states for fractional Laplacians in $\mathbb{R}$}, $Acta$ $Math$.$\mathbf{210}$ (2013), no. 2, 261–318.  \label{Frank Lenzmann Uniqueness of non linear ground states} 




\bibitem{Gassot radially symmetric traveling waves}Gassot, L.
\emph{On the radially symmetric traveling waves for the Schr\"odinger equation on the Heisenberg group}, arXiv:1904.07010. (2019) \label{Gassot radially symmetric traveling waves} 

\bibitem{gerard profile decomposition}G\'erard, P.
\emph{Description du d\'efaut de compacit\'e de l’injection de Sobolev}, ESAIM Control Optim. Calc. Var., 3 :213–233 (electronic),(1998) \label{gerard profile decomposition} 

\bibitem{GerardGrellier1}G\'erard, P., Grellier, S. \emph{The cubic Szeg\H{o} equation}, Ann. Sci. l'\'Ec. Norm. Sup\'er. (4) 43 (2010), 761-810\label{gerardgrellier1}

\bibitem{GG1}G\'erard, P., Grellier, S., \emph{The cubic Szeg\H{o} equation and
Hankel operators}, volume 389 of Astérisque. Soc. Math. de France,
(2017).\label{Gerard grellier book cubic szego equation and hankel operators}

\bibitem{GG1}G\'erard, P., Lenzmann, E., Pocovnicu, O., Rapha\"el, P., \emph{A two-soliton with transient turbulent regime for the cubic half-wave equation on the real line}, Annals of PDE, 4(7) :166, 2018.\label{Gerard Lenzmann Pocovnicu Raphael A two-soliton with transient turbulent regime}



\bibitem{Glassey}Glassey, R., \emph{On the blow up of solutions to the Cauchy problem for non linear Schr\"odinger operators}, J. Math. Phys. 8 (1977), 1794-1797.\label{Glassey, R On the blow up of }

 
 
\bibitem{Hmidi--Keraani application}Hmidi, T., Keraani, S.
\emph{Blowup theory for the critical nonlinear Schrödinger equations revisited}, Int. Math. Res. Not. IMRN 46 (2005), 2815–2828 \label{Hmidi--Keraani application}  
 

\bibitem{Hmidi--Keraani Blowup theory profile decomposition}Hmidi, T., Keraani, S.
\emph{Remarks  on  the  blow-up  for  the $L^2-$critical  nonlinear  Schr\"odinger  equations}, SIAM J. Math. Anal., 38 (2006), no.4, 1035-1047 \label{Hmidi--Keraani Blowup theory profile decomposition} 


 





\bibitem{Lenzmann sok rearrangement of fourier modes}Lenzmann, E., Sok, J.
\emph{A sharp rearrangement principle in Fourier space and symmetry results for PDEs with arbitrary order}, arXiv:1805.06294 (2018) \label{Lenzmann sok rearrangement of fourier modes} 

\bibitem{Lions concentration compactness 1}Lions, P.-L.,
\emph{The concentration-compactness principle in calculus of variations. The locally compact case. Part 1}, Ann. Inst. Henri Poincar\'e, Analyse non lin\'eaire 1 (1984) 109–145 \label{Lions concentration compactness 1} 

\bibitem{Lions concentration compactness 2}Lions, P.-L.,
\emph{The concentration-compactness principle in calculus of variations. The locally compact case. Part 2}, Ann. Inst. Henri Poincar\'e, Analyse non lin\'eaire 1 (1984) 223–283\label{Lions concentration compactness 2} 

\bibitem{Merle--Raphael On universality of blow-up profile L2 critical schrodinger equation}Merle, F., Raphael, P.
\emph{On universality of blow-up profile for $L^2-$critical non linear Schr\"odinger equation}, Invent. math. 156, 565–672 (2004) \label{Merle--Raphael On universality of blow-up profile L2 critical schrodinger equation} 


\bibitem{Perelman On the blow-up phenomenon for the critical nls}Perelman, G.
\emph{On the blow-up phenomenon for the critical nonlinear Schr\"odinger
equation in 1D},  Ann. Henri Poincar\'e 2, 605–673 (2001)  \label{Perelman On the blow-up phenomenon for the critical nls} 


\bibitem{pocovnicu Traveling waves for the cubic Szego eq}Pocovnicu, O.
\emph{Traveling waves for the cubic Szeg\H{o} equation on the real line}, Anal. PDE 4 no. 3 (2011), 379 -404  \label{pocovnicu Traveling waves for the cubic Szego eq} 

\bibitem{Stein SINGULAR integral}Stein, E.M.
\emph{Singular Integrals and Differentiability Properties of Functions}, Princeton Mathematical Series,  Princeton University Press (1970) \label{Stein Singular Integrals and Differentiability Properties of Functions}

\bibitem{Sun Long time behavior of NLS Szego equation}Sun, R.
\emph{Long time behavior of the NLS-Szeg\H{o} equation}, Dyn. Partial. Differ. Equ. 16 (2019), No.4, 325-357 \label{Sun Long time behavior of NLS Szego equation}

\bibitem{Weinstein nls sharp interpolation estimates}Weinstein, M.I.
\emph{Nonlinear Schr\"odinger equations and sharp interpolation estimates}, Comm. Math. Phys. $\mathbf{87}$(1983), 567-576 \label{Weinstein nls sharp interpolation estimates}
\end{thebibliography}
\end{document}